\documentclass[a4paper,11pt]{article}
\usepackage[english]{babel}
\usepackage[latin1]{inputenc}
\usepackage[T1]{fontenc}
\usepackage{amsthm}
\usepackage{amsmath}
\usepackage{amssymb}
\usepackage{xypic}
\usepackage{mathrsfs}
\usepackage{braket}
\usepackage{xcolor}
\usepackage{hyperref}
\usepackage{graphicx,subfig,floatrow}
\usepackage[nottoc]{tocbibind}
\usepackage{caption}
\usepackage{tikz}
\usetikzlibrary{arrows,shapes,positioning,mindmap,trees,automata,decorations.markings}
\usepackage{arydshln}
\usepackage{mathtools}

\hypersetup{
	colorlinks=false,
	allbordercolors=white,
}

\theoremstyle{definition}
\newtheorem{definition}{Definition}[section]

\theoremstyle{plain}
\newtheorem{teo}[definition]{Theorem}

\newtheorem{cor}[definition]{Corollary}
\newtheorem{lem}[definition]{Lemma}
\newtheorem{conj}[definition]{Conjecture}

\newtheorem{claim}{Claim}

\newcommand{\numberset}{\mathbb}

\newcommand{\Z}{\numberset{Z}}

\title{On removable edge subsets in graphs with a nowhere-zero $4$-flow}

\author{D. Mattiolo\thanks{Department of Computer Science, KU Leuven Kulak, 8500 Kortrijk, Belgium. Email: \href{mailto:davide.mattiolo@kuleuven.be}{davide.mattiolo@kuleuven.be}}}

\begin{document}
	\maketitle

	\begin{abstract}
	   A set $R\subseteq E(G)$ of a graph $G$ is \emph{$k$-removable} if $G-R$ has a nowhere-zero $k$-flow. We prove that every graph $G$ admitting a nowhere-zero $4$-flow has a $3$-removable subset consisting of at most $\frac{1}{6}|E(G)|$ edges.
       This gives a positive answer to a conjecture of M.\ DeVos, J.\ McDonald, I.\ Pivotto, E.\ Rollov\'a and R.\ \v S\'amal [\emph{$3$-Flows with large support}, J.\ Comb.\ Theory Ser.\ B 144 (2020), 32-80] in the case of graphs admitting a nowhere-zero $4$-flow.
       
        Moreover, Hoffmann-Ostenhof recently conjectured that every cubic graph with a nowhere-zero $4$-flow has a $4$-removable edge. Bipartite cubic graphs verify this conjecture.
        Our result gives an approximation for Hoffmann-Ostenhof's Conjecture in the non-bipartite case.

        Finally, for cubic graphs, our result implies that every $3$-edge-colorable cubic graph $G$ contains a subgraph $H$ whose connected components are either cycles or subdivisions of bipartite cubic graphs, such that $|E(H)|\ge \frac{5}{6}|E(G)|$.
	\end{abstract}

{\small \noindent {\bf Key words}: removable edges, nowhere-zero flows, flow-continuous maps, $3$-edge-coloring Conjecture.

}
	
\section{Introduction} 

Let $A$ be an (additive) abelian group, $G$ a graph and $\vec{G}$ an orientation of $G$. An $A$-\emph{flow} on $\vec{G}$ is a function $\psi\colon E(\vec{G})\to A$ such that, for every vertex, the sum of all incoming flow values equals the sum of all outgoing ones.
For an integer $k\ge2$, a \emph{$k$-flow} on $\vec{G}$ is a $\Z$-flow $\psi$ on $\vec{G}$ such that $|\psi(z)|\le k-1$ for all $z\in E(\vec{G})$. Moreover, such a flow $\psi$ is called a \emph{nowhere-zero} $k$-flow, or $k$-NZF, if $\psi(z)\ne0$ for all $z\in E(\vec{G})$. It is worth noting that the existence of an $A$-flow $\psi$ does not depend on the chosen orientation $\vec{G}$; indeed if we reverse the orientation of any arc $z\in E(\vec{G})$ and replace the flow value of $\psi(z)$ with its opposite in $A$, we still have an $A$-flow. Thus, we say that $G$ has an $A$-flow if some orientation of $G$ has one.
The \emph{flow number} $\phi(G)$ of $G$ is defined as the smallest number $h$ such that $G$ has a $h$-NZF.
The following result is due to Tutte.

\begin{teo}[\cite{Tutte_imbedding},\cite{Tutte_contribution}]\label{teo:charact_cubic_wrt_flows}
	Let $G$ be a cubic (i.e.\ $3$-regular) graph.
	\begin{itemize}
		\item $G$ is $3$-edge-colorable if and only if $G$ has a $4$-NZF;
		
		\item $G$ is bipartite if and only if $G$ has a $3$-NZF.
	\end{itemize}
\end{teo}

A set $R\subseteq E(G)$ of graph $G$ is called \emph{$k$-removable} if $G-R$ has a $k$-NZF. If $R=\{e\}$, we say that the edge $e$ is $k$-removable.
The main result of this note is the following theorem.

\begin{teo}\label{teo:main}
       Every graph $G$ admitting 
       a $4$-NZF has a $3$-removable set $R$ such that $|R|\le \frac{1}{6}|E(G)|$.
\end{teo}

By Theorems \ref{teo:main} and \ref{teo:charact_cubic_wrt_flows} we deduce the following corollary.

\begin{cor}\label{cor:main_for_cubic}

    Every $3$-edge-colorable cubic graph $G$ contains a subgraph $H$ whose connected components are either cycles or subdivisions of bipartite cubic graphs, such that $|E(H)|\ge \frac{5}{6}|E(G)|$.
\end{cor}

From the proof of Theorem \ref{teo:main} (see Section \ref{sec:proof}), it follows that the subgraph $H$ in Corollary \ref{cor:main_for_cubic} can be chosen to be spanning.

Let $G$ be a bipartite cubic graph with $V(G)=S\cup T$, and let $G'$ be the cubic graph obtained by truncating every vertex of $T$ (i.e.\ replacing each of its vertices with a triangle). In \cite{3-flows_large_supp} it is shown that every $3$-removable set of $G'$ has cardinality at least $\frac{1}{6}|E(G')|$. Since $G$ is $3$-edge-colorable, then $G'$ is also a $3$-edge-colorable cubic graph. Therefore, there are infinitely many cubic graphs for which the the statements of Theorem \ref{teo:main} and Corollary \ref{cor:main_for_cubic} are best possible.

We conclude this section noticing that Theorem \ref{teo:main} follows rather easily for cubic graphs. Indeed, if $G$ is a cubic graph with a $4$-NZF, then by Theorem \ref{teo:charact_cubic_wrt_flows}, $G$ has a proper $3$-edge-coloring $c\colon E(G)\to \{1,2,3\}$. Then, $c^{-1}(1)\cup c^{-1}(2)$ induces a disjoint union of cycles $C_1,\dots,C_n$ in $G$. Let $D$ be an orientation on $c^{-1}(1)\cup c^{-1}(2)$ such that each cycle $C_i$ is a {\it directed circuit} (i.e.\ each of its vertices has one incoming and one outgoing arc). Note that $c^{-1}(1)\cup c^{-1}(3)$ also induces a disjoint union of cycles $C'_1,\dots,C'_t$ in $G$. For any orientation $D'$ of $c^{-1}(1)\cup c^{-1}(3)$, $c^{-1}(1)$ is partitioned into the set $X =\{z\in c^{-1}(1) \colon$ the orientations of $z$ in $D$ and $D'$ are opposite$\}$, and the set $Y = c^{-1}(1)\setminus X$ consisting of the arcs (with color $1$) having the same orientation in both $D$ and $D'$. Choose an orientation $D'$ such that each cycle $C'_j$ is a directed circuit and $|X|\le\frac{1}{2}|c^{-1}(1)|.$ Now, we construct a $3$-flow $\psi$ on $G$ as follows: on the edges of $c^{-1}(1)\cup c^{-1}(2)$ fix the same orientation as in $D$, and on the edges of $c^{-1}(3)$ fix the same orientation as in $D'$; let $\psi(z) = 1$ for all $z\in c^{-1}(2)\cup c^{-1}(3)$, $\psi(z) = 2$ for all $z\in Y$ and $\psi(z)=0$ otherwise. Then, $\psi$ is a $3$-NZF on $G-X$, and thus $X$ is a $3$-removable set such that $|X| \le\frac{1}{2}|c^{-1}(1)|= \frac{1}{6}|E(G)|$.

The proof of the general statement of Theorem \ref{teo:main} will be carried out in Section \ref{sec:proof} using flow continuous maps, which will be introduced in Section \ref{sec:flow_cont_maps}.

\subsection{Consequences}

We present here two consequences of our main result with respect to open problems in the field.

\subsubsection*{$3$-flows with large support}

For a flow $\psi$ on a graph $G$, the \emph{support} supp$\psi$ of $\psi$ is the set of arcs of $G$ having non-zero flow value. In \cite{3-flows_large_supp} it is proved that every $3$-edge-connected graph $G$ has a $3$-flow $\psi$ with $|$supp$\psi|\ge \frac{5}{6}|E(G)|$. The authors also point out that, while the statement seems to hold for $2$-edge-connected graphs, their theorem does not obviously give the same result for $2$-edge-connected graphs. They propose the following conjecture (see Conjecture 8.1 in \cite{3-flows_large_supp}).

\begin{conj}[\cite{3-flows_large_supp}]\label{conj:3flowslargesupp}
    Every $2$-edge-connected graph $G$ has a $3$-flow $\psi$ such that $|$\emph{supp}$\psi|\ge \frac{5}{6}|E(G)|$.
\end{conj}

The best known result, attributed to Tarsi (see Theorem 1.4 in \cite{3-flows_large_supp}), states that every $2$-edge-connected graph $G$ has a $3$-flow $\psi$ such that $|$supp$\psi| \ge\frac{4}{5}|E(G)|.$ 
From Theorem \ref{teo:main}, we get the following corollary, which
gives a positive answer to Conjecture \ref{conj:3flowslargesupp} for ($2$-edge-connected) graphs admitting a $4$-NZF.

\begin{cor}
    Every graph $G$ admitting a $4$-NZF has a $3$-flow $\psi$ such that $|$\emph{supp}$\psi|\ge \frac{5}{6}|E(G)|$.
\end{cor}

\subsubsection*{Removable edges in $3$-edge-colorable cubic graphs}

A conjecture recently proposed by Hoffmann-Ostenhof \cite{HO_3ec_conj} suggests that a $3$-edge-colorable cubic graph has an edge whose deletion leads to a graph homeomorphic to a $3$-edge-colorable cubic graph.
This conjecture can be rephrased in terms of nowhere-zero flows as follows.

\begin{conj}[\cite{HO_3ec_conj}]\label{conj:flow_version}
	A cubic graph admitting a $4$-NZF has a $4$-removable edge.
\end{conj}

It is well known that if an edge $e=uv$ of a graph $G$ is not in a $2$-edge-cut, then $\phi(G-e)\le \phi(G)+1.$ To see this let $\psi$ be a positive $\phi(G)$-NZF on a suitable orientation $\vec{G}$ of $G$ and assume without loss of generality that $e$ is directed from $u$ to $v$ in $\vec{G}$. Since $\psi$ is nowhere-zero, there is a directed path $P$ from $u$ to $v$ in $\vec{G}-e$. Adding the flow value $1$ along $P$ gives a $k$-NZF on $\vec{G}-e$, with $k\le\phi(G)+1.$

It follows that bipartite cubic graphs verify Conjecture \ref{conj:flow_version} and therefore, by Theorem \ref{teo:charact_cubic_wrt_flows}, we only need to study cubic graphs with flow number $4.$


Note that, for all $k\ge2$, a $k$-removable set is also $(k+1)$-removable. Thus, from Theorem~\ref{teo:main} and the above discussion we deduce the following corollary. It gives an approximation to Conjecture \ref{conj:flow_version} for non-bipartite $3$-edge-colorable cubic graphs.

\begin{cor}
    Every bipartite cubic graph has a $4$-removable edge. Moreover, every cubic graph $G$ admitting a $4$-NZF 
    has a non-empty $4$-removable set of cardinality at most $\frac{1}{6}|E(G)|$.
\end{cor}

\section{Flow continuous maps}\label{sec:flow_cont_maps}

In the proof of Theorem \ref{teo:main} we make use of $\Z$-flow-continuous maps.
Flow continuous maps were introduced in \cite{DvNR} as generalization of cycle continuous maps 
in the context of the famous Petersen coloring Conjecture \cite{Jaeg_NZF_problems}.

Let $A$ be an abelian group and fix on $G_1$ and $G_2$ an orientation $\vec{G_1}$ and $\vec{G_2}$, respectively. A map $f\colon E(\vec{G_1})\to E(\vec{G_2})$ is called $A$-\emph{flow-continuous} if, for every $A$-flow $\psi$ on $\vec{G_2}$, $\psi\circ f$ is an $A$-flow on $\vec{G_1}.$ If such a map exists we write $G_1 \succ_A G_2.$ See \cite{Matt, nes_sam} for further reading on flow-continuous maps.

Given an orientation $\vec{G}$ of a graph $G$ and an edge $e\in E(G)$, we denote by $z_e$ the arc of $\vec{G}$ corresponding to an orientation of $e$.
An $A$-flow-continuous map $f\colon E(\vec{G_1})\to E(\vec{G_2})$ induces a map $f_*\colon E(G_1)\to E(G_2)$ by the following rule: $f_*(e_1) = e_2$ if and only if $f(z_{e_1})=z_{e_2}$. 
Note that, for $X\subseteq E(\vec{G_2})$, $\vec{G_2}'=\vec{G_2}-X$ and $\vec{G_1}'=\vec{G_1}-f^{-1}(X)$, then $f\colon E(\vec{G_1}')\to E(\vec{G_2}')$ is still $A$-flow-continuous. Therefore,
the following lemma holds true.

\begin{lem}\label{lem:removal_edges}
    Let $G_1$ and $G_2$ be graphs and $A$ an abelian group. If there is an $A$-flow-continuous map $f\colon E(\vec{G_1})\to E(\vec{G_2})$, then for all $Y\subseteq E(G_2)$, $G_1-f_*^{-1}(Y) \succ_A G_2-Y.$
\end{lem}

The following theorem characterizes the graphs having a $4$-NZF. We denote by $K_4$ the complete graph on four vertices.

\begin{teo}[\cite{DvNR}]\label{teo:4nf_iff_Zfc_to_K4}
	A graph $G$ has $4$-NZF if and only if $G\succ_{\Z} K_4.$
\end{teo}

\section{Proof of Theorem \ref{teo:main}}\label{sec:proof}

Let $G$ be a graph admitting a $4$-NZF. If $\phi(G)\le3$, then the empty set is a $3$-removable set for $G$.
Thus, let $G$ be a graph with $\phi(G)=4$. By Theorem \ref{teo:4nf_iff_Zfc_to_K4}, there is a $\Z$-flow-continuous map $f\colon E(\vec{G})\to E(\vec{K_4})$, where $\vec{G}$ and $\vec{K_4}$ are suitable orientations of $G$ and $K_4$ respectively.


\begin{claim}\label{prop:f_surjective}
    $f$ is surjective.
\end{claim}

\begin{proof}
    By contradiction, assume that there is $z_e\in E(\vec{K}_4)$ such that $z_e\notin$ Im$(f)$. Then, by Lemma \ref{lem:removal_edges}, $G\succ_{\Z} K_4-e$. Note that $K_4-e$ is homeomorphic to $K_2^3$ (the multigraph on two vertices connected by three parallel edges), and thus has a $3$-NZF. Hence it follows that $G$ has a $3$-NZF, a contradiction.
\end{proof}

    By Claim \ref{prop:f_surjective}, we infer that $E(G) = \cup_{i=1}^6 f_*^{-1}(e_i)$ with $f_*^{-1}(e_i)\ne \emptyset$ for all $i\in\{1,\dots,6\}$, where $E(K_4) = \{e_1,\dots,e_6\}$. Thus, there is $t \in \{1,\dots, 6\}$ such that $1\le|f_*^{-1}(e_t)|\le \frac{1}{6}|E(G)|.$ Let $R=f_*^{-1}(e_t)$. Since $1\le|R|\le \frac{1}{6}|E(G)|$, then $H=G-R$ is a proper non-trivial subgraph of $G$. By Lemma \ref{lem:removal_edges}, $H\succ_{\Z} K_4-e_t$ and thus, since $K_4-e_t$ has a $3$-NZF
    , we infer that $H$ has a $3$-NZF. Hence, $R$ is a $3$-removable set such that $|R| \le \frac{1}{6}|E(G)|.$

\section{Acknowledgments}

The author would like to thank G.\ Mazzuoccolo and E.\ Steffen for their remarks which helped improving the presentation of the manuscript.
The author acknowledges support from a Postdoctoral Fellowship of the Research Foundation Flanders (FWO), grant 1268323N.

\end{document}